\documentclass[11pt]{amsart}
\usepackage{amsmath}
\usepackage{amssymb}
\usepackage{amsthm}
\usepackage{array}
\usepackage{xy}
\usepackage[pdftex]{graphicx}
\usepackage{hyperref}
\usepackage{color}
\usepackage{transparent}
\usepackage{latexsym}

\numberwithin{equation}{section}

\newtheorem{thm}{Theorem}[section]
\newtheorem{cor}[thm]{Corollary}
\newtheorem{lem}[thm]{Lemma}

\newtheorem{rem}[thm]{Remark}

\newcommand{\R}{\mathbb{R}}

\newcommand{\Lap}{\Delta}

\newcommand{\into}{\rightarrow}

\newcommand{\grad}{\nabla}

\newcommand{\norm}[1]{\left\Vert#1\right\Vert}

\begin{document}
\title[A quantitative log-Sobolev inequality for a two parameter family of functions]{A quantitative log-Sobolev inequality for a two parameter family of functions}
\author[E. Indrei]{E. Indrei}
\author[D. Marcon]{D. Marcon}

\date{}

\maketitle


\def\signei{\bigskip\begin{center} {\sc Emanuel Indrei\par\vspace{3mm}
Department of Mathematics\\
The University of Texas at Austin\\
1 University Station, C1200\\
Austin TX 78712, USA\\
email:} {\tt eindrei@math.utexas.edu}
\end{center}}

\def\signdm{\bigskip\begin{center} {\sc Diego Marcon\par\vspace{3mm}
Center for Mathematical Analysis,\\
Geometry, and Dynamical Systems\\
Departamento de Matem\'atica\\
Instituto Superior T\'ecnico\\
Lisboa 1049-001, Portugal\\
email:} {\tt dmarcon@math.utexas.edu}
\end{center}}

\begin{abstract}
We prove a sharp, dimension-free stability result for the classical logarithmic Sobolev inequality for a two parameter family of functions. Roughly speaking, our family consists of a certain class of log $C^{1,1}$ functions. Moreover, we show how to enlarge this space at the expense of the dimensionless constant and the sharp exponent. As an application we obtain new bounds on the entropy. 
\end{abstract}

\section{Introduction}
\subsection{Overview}
Sobolev-type inequalities are central tools in analysis. The so-called logarithmic Sobolev inequalities appear in various branches of statistical mechanics, quantum field theory, and mathematical statistics. For example, the Gaussian log-Sobolev inequality is equivalent to Nelson's hypercontractive inequality and one may use log-Sobolev inequalities to show the stabilization of the Glauber-Langevin dynamic stochastic models for the Ising model with real spins, see for instance \cite{Roy, Gro}. 

Moreover, they are also useful in partial differential equations and Riemannian geometry. Indeed, they showed up in Perelman's work on the Ricci flow and the Poincar\`{e} conjecture \cite{Per}. While there is a large body of literature available on these inequalities, there are few corresponding stability results and this is currently an active area of research. Figalli, Maggi, and Pratelli \cite{Fi5} have recently addressed the stability problem for the anisotropic $1$-log-Sobolev inequality; however, stability for the Gaussian log-Sobolev inequality -- the classical version attributed to Stam \cite{St}, Federbush \cite{Fe}, and Gross \cite{Gro}  -- is still open. 

In this paper, we address this problem for a two parameter family of functions. Our approach involves techniques from optimal transport theory. Indeed, this theory has proven useful in producing sharp geometric and functional inequalities (see e.g. \cite{Fi}, \cite{FiI}).

\subsection{Main result}
The classical Gaussian log-Sobolev inequality states that for smooth, positive functions
$$\int_{\mathbb{R}^n} f(x) \log f(x) d\gamma (x) - ||f||_{L^1(d\gamma)} \log\big (||f||_{L^1(d\gamma)}\big) \leq \frac{1}{2} \int_{\mathbb{R}^n} \frac{|\nabla f (x)|^2}{f(x)}  \ d\gamma (x),$$
where $d\gamma := (2\pi)^{-n/2}e^{-|x|^2/2}dx$ is the standard Gaussian measure. The right hand side of the inequality is known as the \textit{Fisher information} and is often denoted by $I(f)$ whereas the left hand side is the \textit{entropy} and represented by $Ent(f)$. It is well-known that equality holds if and only if $f$ is log linear (i.e. $f(x)=e^{a\cdot x+b}$).
For $\epsilon>0$ and $M>0$, consider the family of functions:
$$\mathcal{F}(\epsilon, M) := \Bigl\{e^{-h}: (-1+\epsilon)\leq D^2 h \leq M \Bigr\},$$ and denote the \textit{log-Sobolev deficit} by
\begin{align*}
\delta(f)&:=\frac{1}{2} \int_{\mathbb{R}^n} \frac{|\nabla f (x)|^2}{f(x)}  \ d\gamma (x) - \int_{\mathbb{R}^n} f(x) \log f(x) d\gamma (x) + ||f||_{L^1(d\gamma)} \log\big (||f||_{L^1(d\gamma)}\big)\\
&=\frac{1}{2}I(f) - Ent(f).
\end{align*}
Note that $\delta \geq 0$ by the log-Sobolev inequality. The main result of this paper is the following theorem:

\begin{thm} \label{logsob}
There exists an explicit dimensionless constant $C=C(\epsilon, M)>0$ so that for all $f \in \mathcal{F}(\epsilon, M)$, $$W_2\Big(f(\cdot)e^{-(\langle \mu, \cdot \rangle + |\mu|^2/2 + \log(m))} d\gamma, d\gamma \Big) \leq C\delta(f/m)^{\frac{1}{2}},$$ where $W_2$ is the Wasserstein metric, $m=||f||_{L^1(d\gamma)}$, and $\mu$ is the barycenter of f.
\end{thm}

Our theorem gives a quantitative way of measuring how far an admissible function is from attaining equality in the log-Sobolev inequality as measured with the Wasserstein metric. The proof of Theorem \ref{logsob} is achieved by showing its equivalence to the following corollary (see \S \ref{S3}):

\begin{cor} \label{logsob0}
There exists an explicit dimensionless constant $C=C(\epsilon, M)>0$ so that for all $f \in \mathcal{F}(\epsilon, M)$ with unit mass and zero barycenter, $$W_2(fd\gamma, d\gamma) \leq C \delta(f)^{\frac{1}{2}}.$$
\end{cor}

Although $\mathcal{F}(\epsilon, M)$ has a special structure, our results could be seen as a first step towards a sharp, general, dimension-free stability result for the Gaussian log-Sobolev inequality. In fact, by modifying our proof of Theorem \ref{logsob}, the class of admissible functions $\mathcal{F}(\epsilon, M)$ may be enlarged at the expense of the dimensionless constant and the sharp exponent. More specifically, thanks to a recent result of Kolesnikov \cite{Kol}, one may replace the upper $L^\infty$ assumption on the Hessian of the logarithm of admissible functions with an $L^r$ estimate.

\begin{thm} \label{logsob2}
If $\delta \leq 1$ and $r >1$, then there exist explicit constants $C=C(\epsilon, M, n)>0$ and $\beta=\beta(r)>0$ so that for all
\begin{align*}
f \in \mathcal{\tilde F}(\epsilon, M, r):=\Biggl\{e^{-h}: &(-1+\epsilon)\leq D^2 h, \int_{{\mathbb{R}^n}} ||(D^2 h+Id)_+||^{r}fd\gamma \leq M \Biggr\}
\end{align*}
with unit mass and zero barycenter, $$W_2(f d\gamma, d\gamma) \leq C \delta(f)^{\beta}.$$ Moreover, one may take $\beta=\frac{r-1}{2(2r-1)}$.
\end{thm}
\noindent We remark that one may remove the unit mass and zero barycenter assumptions in Theorem \ref{logsob2} and prove an analogous result as in Theorem \ref{logsob}.

The paper is organized as follows: in \S \ref{S2}, we collect some preliminary results from the literature which will be used in our proofs. In \S \ref{S3} we prove Theorems \ref{logsob} and \ref{logsob2} and show that the $\frac{1}{2}$ exponent is sharp. Last, in \S \ref{S4} we show how to obtain bounds on the entropy in terms of the deficit (see Corollary \ref{bet1}) and derive an improved log-Sobolev inequality for our function class (see Remark \ref{bet2}).

\section{Preliminaries}
\label{S2}

Our proof of Theorem \ref{logsob} exploits Cordero-Erausquin's optimal mass transfer proof of the log-Sobolev inequality \cite{Co}. For the reader's convenience and to simplify the presentation of our proof, we include his proof in this section along with statements of other results from the literature which will be useful for our purpose. We recall that given a smooth, positive function $f:\R^n \to \R$ normalized to have unit mass with respect to the Gaussian measure $d\gamma$, Brenier's theorem yields the existence of a convex function $\phi:\R^n \to \R$ such that its gradient $\grad \phi$ is the optimal transport map between $fd\gamma$ and $d\gamma$: i.e.$$\grad \phi_{\#} (fd\gamma) = d\gamma$$ and $$\int_{\R^n} |x-\grad \phi(x)|^2 d\mu(x) = \inf_{T_{\#} (fd\gamma) = d\gamma}  \int_{\R^n} |x-T(x)|^2 d\mu(x).$$ Moreover, $\phi$ satisfies $fd\gamma$ -- a.e. the Monge-Amp\`ere equation $$f(x)e^{-|x|^2/2}=\det(D^2 \phi)e^{-|\nabla \phi (x)|^2/2}.$$ For appropriate definitions from transport theory, we refer the reader to \cite{Vi2} (see also the introduction in \cite{Co} for a short and clear overview).

\begin{thm} \label{C} (log-Sobolev) Let $f$ be a smooth, positive function on $\mathbb{R}^n$ normalized to have unit mass with respect to the Gaussian measure $d\gamma$. Then,
$$\int_{\mathbb{R}^n} f(x) \log f(x) d\gamma (x)- ||f||_{L^1(d\gamma)} \log\big (||f||_{L^1(d\gamma)}\big)  \leq \frac{1}{2} \int_{\mathbb{R}^n} \frac{|\nabla f (x)|^2}{f(x)}  \ d\gamma (x).$$
\end{thm}

\begin{proof}[Proof as given by Cordero-Erausquin \cite{Co}]
Without loss of generality, assume $||f||_{L^1(d\gamma)}=1$. Let $\nabla \Phi$ be the Brenier map between $f d\gamma$ and $\gamma$, and set $\theta(x):=\Phi(x)-\frac{1}{2} |x|^2$ so that $\nabla \Phi(x)=x+\nabla \theta(x)$. It follows that $Id+D^2 \theta \geq 0,$ where $Id$ is the identity matrix. The Monge-Amp\`{e}re equation reads:
$$ f(x)e^{-|x|^2/2}=\det(Id+D^2 \theta)e^{-|x+\nabla \theta (x)|^2/2},$$ $fd\gamma$ -- a.e.
Taking the logarithm of both sides, the above equation may be rewritten as:
\begin{align}
\log f(x) &= -\frac{1}{2}|x+\nabla \theta(x)|^2+\frac{1}{2}|x|^2+\log \det(Id+D^2 \theta) \nonumber \\
&=-x\cdot \nabla \theta(x) - \frac{1}{2}|\nabla \theta(x)|^2+\log \det(Id+D^2 \theta) \nonumber \\
&\leq -x \cdot \nabla \theta(x)-\frac{1}{2}|\nabla \theta(x)|^2+ \Lap \theta(x), \label{corderopf}
\end{align}
where the last inequality follows from the fact that $\log(1+t) \leq t$, for $t\geq-1$ (here, $\log$ is the natural logarithm). Integrating with respect to $fd\gamma$ and using integration by parts, it follows that
\begin{align*}
\int_{{\R}^n} f \log f d\gamma &\leq \int_{{\R}^n} f[\Lap \theta-x\cdot \nabla \theta]d\gamma-\int_{{\R}^n} \frac{1}{2}|\nabla \theta|^2 f d\gamma\\
&= -\int_{{\R}^n} \nabla \theta \cdot \nabla f d \gamma-\int_{{\R}^n} \frac{1}{2}|\nabla \theta|^2 f d\gamma \\
&=-\int_{{\R}^n} \frac{1}{2} \Bigg | \sqrt{f} \nabla \theta(x)+\frac{\nabla f(x)}{\sqrt{f}} \Bigg|^2 d\gamma(x)+ \frac{1}{2} \int_{{\R}^n} \frac{|\nabla f|^2}{f} d \gamma\\
&\leq \frac{1}{2} \int_{{\R}^n} \frac{|\nabla f|^2}{f} d \gamma.
\end{align*}
 \end{proof}

 \begin{rem} \label{cordstab}
Note that by the proof of Theorem \ref{C} (more specifically, from (\ref{corderopf})), if $||f||_{L^1(d\gamma)}=1$, then
$$
\delta(f)  \geq \int_{{\R}^n} f \bigg( \Lap \theta - \log \det (Id + D^2 \theta ) \bigg) \ d\gamma.
$$

 \end{rem}

\noindent Next, we state the following two theorems of Kolesnikov \cite[Theorems 6.1 \& 7.4 ]{Kol}. The first generalizes Caffarelli's contraction theorem \cite{Caffcont}:

\begin{thm} \label{Kol}
Let $\mu=e^{-V}dx$ and $\nu=e^{-W}$ be probability measures on $\mathbb{R}^d$ and let $T=\nabla \Phi$ be the corresponding optimal transport map. If $D^2 W\geq KId$, then for any $1\leq r \leq \infty$, $$K||\Phi_{ee}^2||_{L^r(\mu)} \leq ||(V_{ee})_{+}||_{L^r(\mu)}.$$
\end{thm}

\begin{thm} \label{Kol2}
Let $\mu=e^{-V}dx$ and $\nu=e^{-W}$ be probability measures on $\mathbb{R}^d$ and let $T=\nabla \Phi$ be the corresponding optimal transport map. If $D^2 W\geq KId$, then for any $r\geq1$,

$$K \Big(\int_{{R}^n} ||D^2 \Phi||^{2r} d\mu \Big)^{\frac{1}{r}} \leq \Big(\int_{{R}^n} ||(D^2V)_+||^rd\mu\Big)^{\frac{1}{r}},$$ where $||\cdot||$ denotes the standard induced matrix norm.
\end{thm}

\noindent Last, we shall make use of a well-known Poincar\`{e} type inequality for log convex measures, see \cite[Theorem 2]{AD} (see also \cite{BD}):

\begin{thm}
Let $\mu=e^{-V}$ be uniformly log convex with unit mass and $u \in H^1(\mu)$. Then, $$\frac{1}{2}\Bigl(\frac{p}{p-1}\Bigr)^2 \Bigl[\int_{\mathbb{R}^n} u^2 d\mu-\Bigl(\int_{\mathbb{R}^n} |u|^{\frac{2}{p}}d\mu \Bigr)^{2(p-1)} \cdot \Bigl(\int_{\mathbb{R}^n} u^2 d\mu \Bigr)^{\frac{2}{p}-1} \Bigr] \leq \frac{2}{\lambda_1} \int_{\mathbb{R}^n} |\nabla u|^2 d\mu,$$ where $p \in (1,2]$ and $$\lambda_1:= \inf_{x\in \mathbb{R}^n \xi \in \mathbb{S}^{n-1}} \langle D^2V(x) \xi, \xi\rangle>0.$$
\end{thm}

\noindent For our purposes, we will need the above theorem with $p=2$.

\begin{cor} \label{cor} Let $\mu=e^{-V}$ be uniformly log convex with unit mass and $u \in H^1(\mu)$. Then $$\int_{{R}^n} |u-\bar{u}|^2 d\mu \leq C(\mu)  ||\nabla u||_{L^2(\mu)}^2,$$ where $\bar{u}:= \int_{{R}^n} u d\mu$ $C(\mu):=\frac{1}{\lambda_1}$.
\end{cor}

\section{Proofs}
\label{S3}

\subsection{$L^\infty$ case}

\begin{proof}[Proof of Theorem \ref{logsob}]
First, assume $||f||_{L^1(d\gamma)}=1$ and that $f$ has barycenter equal to zero with respect to $d\gamma$. Let $T=\nabla \Phi$ be the optimal transport between $fd\gamma$ and $d\gamma$ and set $\theta(x):=\Phi(x)-\frac{|x|^2}{2}$ (recall that $\Phi$ is convex). By Remark \ref{cordstab}, we have
$$
\delta(f)  \geq \int_{{\R}^n} f \bigg( \Lap \theta - \log \det (Id + D^2 \theta ) \bigg) \ d\gamma.
$$
We can express $\Lap \theta - \log \det (Id + D^2 \theta )$ as
\begin{align*}
\Lap \theta - \log \det (Id + D^2 \theta) & = \sum_{i=1}^n \lambda_i - \log \bigg(\prod_{i=1}^n (1 + \lambda_i) \bigg)\\
    & = \sum_{i=1}^n (\lambda_i - \log (1 + \lambda_i)),
\end{align*}
where $\{ \lambda_i \}_{i=1}^n$ are the eigenvalues of $D^2 \theta$.
Define $g: (-1, \infty) \into \mathbb{R}$ by $g(t) := t - log (1+ t)$. For some $c>0$ small enough, it follows that
$$
g(t) \ge \phi (t) := c \min \{t^2, |t| \},
$$
(this is easily seen by noting that $g$ is quadratic at the origin and linear at infinity). Hence,
\begin{align*}\label{main1}
\delta(f) & \geq \sum_{i=1}^n \int_{{\R}^n} \phi(|\lambda_i|) f \ d\gamma \\
&\geq c\sum_{i=1}^n \bigg( \int_{\{x \in \R^n : \ |\lambda_i(x)| \ge 1\}} |\lambda_i| f \ d\gamma +  \int_{ \{ x \in \R^n : \ |\lambda_i(x)| \le 1 \} } |\lambda_i|^2 f \ d\gamma \bigg).
\end{align*}
Let $\mu_i:=1+\lambda_i\geq0$ be the eigenvalues of $DT=D^2\Phi$ with corresponding eigenvectors $v_i$ (recall that $\Phi$ is convex, so $\mu_i \geq 0$). It follows that $$\mu_i(x) = \bigl \langle v_i(x), DT(x) v_i(x) \bigr \rangle \leq \sup_{e \in \mathbb{S}^{n-1}}||\Phi_{ee}||_\infty.$$ Now, we apply Theorem \ref{Kol} with $V:=h+\frac{|x|^2}{2}$ and $W:=\frac{|x|^2}{2}$ and note that $D^2 V = D^2 h+I \leq (M+1)I$ to obtain $$\sup_{e \in \mathbb{S}^{n-1}}||\Phi_{ee}||_\infty \leq \sup_{e \in \mathbb{S}^{n-1}}\sqrt{||V_{ee}||_\infty} \leq \sqrt{M+1}.$$
Therefore, $$||\lambda_i||_{\infty} \leq C_M:=\max\{1,\sqrt{M+1}-1\}.$$
Set $E_i:=\{x \in \R^n : \ |\lambda_i(x)| \ge 1\}$ and $\mu_f:=fd\gamma$ so that
\begin{equation}
\int_{E_i} |\lambda_i|^2 \ d\mu_f \le \norm{\lambda_i}_{\infty} \int_{E_i} |\lambda_i| \ d\mu_f \le C_M \int_{E_i} |\lambda_i| \ d\mu_f;
\end{equation} thus,
\begin{equation} \label{bzx1}
\delta(f) \ge (c/ C_M)\int_{{\R}^n} \sum_{i=1}^n |\lambda_i|^2 \ d\mu_f = (c/ C_M) \int_{{\R}^n} ||D(T(x) - x)||_{HS}^2 \ d\mu_f,
\end{equation}
where $||\cdot||_{HS}$ denotes the Hilbert-Schmidt norm. Now let $T=(T^1,T^2,\ldots,T^n)$ and note that since $T_{\#}\mu_f = d\gamma$, we have $$\int_{\mathbb{\R}^n} (T(x)-x))d \mu_f = -\int_{{\R}^n} xd \mu_f(x)= 0.$$ By applying Poincar\'{e} (Corollary \ref{cor})  to $T^i(x)-x_i$, we obtain
\begin{align*}
\int_{{\R}^n} |T(x) - x|^2 \ d\mu_f (x)&= \sum_i \int_{{\R}^n} |T^i(x)-x_i|^2d\mu_f(x)\\
&\leq C(\mu_f) \sum_i \int_{{\R}^n} |\nabla(T^i(x)-x_i)|^2d\mu_f(x) \\
&= C(\mu_f) \int_{{\R}^n} \sum_{i,j} |T_{x_j}^i - \delta_{ij}|^2d\mu_f(x).
\end{align*}
Next, let $(a_{ij}(x))$ be the tensor $DT(x)-Id$ and note $a_{ij}(x)=T_{x_j}^i-\delta_{ij}$; in particular,
$$\int_{{\R}^n} \sum_{i,j} |T_{x_j}^i - \delta_{ij}|^2d\mu_f(x)=\int_{{\R}^n} ||DT-Id||_{2}^2 d\mu_f(x)=\int_{{\R}^n} ||DT-Id||_{HS}^2 d\mu_f(x).$$
Combining this information with (\ref{bzx1}), we obtain $$W_2^2(f d\gamma, d\gamma) = \int_{{R}^n} |T(x) - x|^2 \ d\mu_f (x) \leq \frac{1}{c} C(\mu_f)C_M \delta(f).$$ Now $\lambda_1 \geq \epsilon$ since $f=e^{-h} \in \mathcal{F}(\epsilon, M)$ and so by applying Corollary \ref{cor}, we obtain $C(\mu_f) \leq \frac{1}{\epsilon}$; setting $C=C(\epsilon, M):=\sqrt{\frac{1}{\epsilon} (\max\{1,\sqrt{1+M}-1\})}$ completes the proof when $||f||_{L^1(d\gamma)}=1$ and $f$ has zero barycenter with respect to $d\gamma$. Next, assume $$\mu = \int_{\R^n} x f(x) \ d\gamma \neq 0.$$ Define $\hat f (x):= f(x+\mu) e^{-\big(\mu \cdot x + \frac{|\mu|^2}{2}\big)}$. It is easy to see that $$\int_{\R^n} x \hat f(x) \ d\gamma = \int_{\R^n} (x-\mu) f(x) \ d\gamma = 0,$$ and $||f||_{L^1(d\gamma)}=||\hat f||_{L^1(d\gamma)}$. Therefore, applying the previous argument to $\hat f / ||\hat f||_{L^1(d\gamma)}$ yields
\begin{equation} \label{yaah}
W_2\big(\big(\hat f/||\hat f||_{L^1(d\gamma)}\big)d\gamma, d\gamma\big) \leq C \delta\big(\hat f/||f||_{L^1(d\gamma)}\big)^{\frac{1}{2}}.
\end{equation}

\noindent Next we compute $\delta(\hat f)$. To do this, suppose for the moment that $||f||_{L^1(d\gamma)}=1$; then,

\begin{equation*}
\begin{split}
\int_{\R^n} \hat f \log \hat f \ d\gamma & = \int_{\R^n}  f(x + \mu) e^{-\big(\mu \cdot x + \frac{|\mu|^2}{2}\big)} \log f(x+\mu) \ d\gamma(x) \\
     & \ \ \ \ - \int_{\R^n} \bigg(\mu \cdot x + \frac{|\mu|^2}{2}\bigg) f(x+\mu) e^{-\big(\mu \cdot x + \frac{|\mu|^2}{2}\big)} \ d\gamma(x) \\
     & = Ent(f) - \mu \cdot \int_{\R^n}  x f(x+\mu)  \ d\gamma(x+\mu)  - \frac{|\mu|^2}{2} \\
     & = Ent(f) - \frac{|\mu|^2}{2},
\end{split}
\end{equation*} where we used that the barycenter of $\hat f$ is zero with respect to $d\gamma$; moreover,

\begin{equation*}
\begin{split}
\int_{\R^n} \frac{|\grad \hat f|^2}{\hat f} \ d\gamma & = \int_{\R^n}  \Bigg| \frac{\grad f(x + \mu) e^{-\big(\mu \cdot x + \frac{|\mu|^2}{2}\big)} - \mu f(x+\mu) e^{-\big(\mu \cdot x + \frac{|\mu|^2}{2}\big)}}{f(x+\mu)e^{-\big(\mu \cdot x + \frac{|\mu|^2}{2}\big)} }\Bigg|^2 f(x+\mu) \ d\gamma(x+ \mu) \\
     & = \int_{\R^n}  \bigg| \frac{\grad f(x + \mu)}{f(x+\mu)} - \mu \bigg|^2 f(x+\mu) \ d\gamma(x+ \mu) \\
     & = \int_{\R^n}  \bigg| \frac{\grad f(x )}{f(x)} - \mu \bigg|^2 f(x) \ d\gamma(x) \\
     & = I(f) -2\mu \cdot \int_{\R^n} \grad f (x) \ d\gamma + |\mu|^2 = I(f) - |\mu|^2,
\end{split}
\end{equation*}
where we used integration by parts in the last equality to deduce $\mu=\int_{\R^n} \grad f (x) \ d\gamma$. The above considerations readily imply that if $||f||_{L^1(d\gamma)}=1$, then $\delta\big(\hat f \big)=\delta(f)$, and the general case follows from the fact that $\delta$ is positively $1$-homogeneous; combining these facts with (\ref{yaah}) concludes the proof.

\end{proof}

\begin{rem}
We note that the admissible functions are of the form $e^{-h}$, where $h$ is semi-concave and semi-convex and the opening of the parabolas touching from above and below depend on the parameters $\epsilon$ and $M$. Therefore, the logarithm of the admissible functions have $C^{1,1}$ norms depending on these two parameters. In the proof, the upper bound on the Hessian of $h$ was used to go from $L^1$ to $L^2$ (via Caffarelli/Kolesnikov), whereas the lower bound was used to apply Poincar\'e.
\end{rem}

\begin{rem}
By considering a family of rescaled Gaussian measures it is not difficult to see that the exponent $\frac{1}{2}$ is sharp, see \S \ref{subsect: sharp}.
\end{rem}

\subsection{$L^r$ case}

\begin{proof}[Proof of Theorem \ref{logsob2}]
Let $T=\nabla \Phi$ be the optimal transport between $fd\gamma$ and $d\gamma$ and set $\theta(x):=\Phi(x)-\frac{|x|^2}{2}$. By Remark \ref{cordstab}, we have

\begin{equation*} \label{sdf}
\delta(f)  \geq \int_{{\R}^n} f \bigg( \Lap \theta - \log \det (Id + D^2 \theta ) \bigg) \ d\gamma.
\end{equation*}
We express $\Lap \theta - \log \det (I + D^2 \theta )$ as
\begin{align*}
\Lap \theta - \log \det (Id + D^2 \theta) & = \sum_{i=1}^n \lambda_i - \log \bigg(\prod_{i=1}^n (1 + \lambda_i) \bigg)\\
    & = \sum_{i=1}^n (\lambda_i - \log (1 + \lambda_i)),
\end{align*}
where $\{ \lambda_i \}_{i=1}^n$ are the eigenvalues of $D^2 \theta$.
Define $g: (-1, \infty) \into \mathbb{R}$ by $g(t) := t - log (1+ t)$ and
\begin{displaymath}
   \phi(t):= \left\{
     \begin{array}{lr}
      \frac{t^2}{2}, & \hskip .05in -1\leq t \leq 0\\
       t-ln(1+t), \hskip .05in & \hskip 0.05in t\geq0.
     \end{array}
   \right.
\end{displaymath}
Note that $\phi(t)=\phi(|t|)$ is convex and $g(t)\geq \phi(t)$. By Jensen's inequality, we obtain
\begin{align*}
\delta(f) & \geq \sum_{i=1}^n \int_{{\R}^n} \phi(|\lambda_i|) f \ d\gamma \\
&\geq \sum_{i=1}^n \phi\Big(\int_{{\R}^n} |\lambda_i| f d\gamma\Big). 
\end{align*}
Since $\delta \leq 1$, it follows that $$\phi\Big(\int_{{\R}^n} |\lambda_i| f d\gamma\Big) \geq c\Big(\int_{{\R}^n} |\lambda_i| f d\gamma\Big)^2,$$
for a small enough constant $c>0$: in fact, $t-\log(1+t) \geq ct^2$ for $0\leq t \leq \frac{1-2c}{2c}$ and so if $t^*$ satisfies $t^*-\log(1+t^*)=1$, we can pick $c$ so that $t^*=\frac{1-2c}{2c}$ \big(i.e. $c=\frac{1}{2(1+t^*)}\big)$; hence,
\begin{equation*}
\delta(f)\geq c\sum_{i=1}^n \Big(\int_{{\R}^n} |\lambda_i| f d\gamma\Big)^2.
\end{equation*}

\noindent Now we apply Theorem \ref{Kol2} with $V:=h+\frac{|x|^2}{2}$ and $W:=\frac{|x|^2}{2}$ so that $$\int_{{\R}^n} ||D^2 \Phi||^{2r} fd\gamma \leq \int_{{\R}^n} ||(D^2 V)_+||^{r} fd\gamma=\int_{{\R}^n} ||(D^2 h+Id)_+||^{r} fd\gamma \leq M.$$ Thus,
\begin{align*}
\int_{{\R}^n} |\lambda_i|^{2r} f d\gamma&\leq \int_{{\R}^n} ||D^2\Phi - Id||^{2r} fd\gamma \label{koles} \\
&\leq \int_{{\R}^n} (||D^2 \Phi|| +1)^{2r} fd\gamma \nonumber \\
&\leq 2^{2r-1}\Big(\int_{{\R}^n} ||D^2 \Phi||^{2r} fd\gamma+1\Big)\nonumber \\
&\leq 2^{2r-1}\Big(M+1\Big):=C(r, M), \nonumber
\end{align*}


\noindent Next, for $p>2$, a standard interpolation inequality yields

\begin{equation*} \label{inter}
\Big (\int_{{\R}^n} |\lambda_i|^2 f d\gamma\Big)^{\frac{1}{2}} \leq ||\lambda_i||_{L^1(fd\gamma)}^\theta||\lambda_i||_{L^p(fd\gamma)}^{1-\theta},
\end{equation*}
where $$\frac{1}{2}=\frac{(1-\theta)}{p}+\theta.$$ Thus, $p=\frac{1-\theta}{\frac{1}{2}-\theta}$ and we may pick $\theta=\frac{r-1}{2r-1}$ so that $p=2r$. Hence,
\begin{equation*} \label{inter2}
\Big(\int_{{\R}^n} |\lambda_i|^2 f d\gamma\Big)^{\frac{1}{\theta}} \leq \Big(\int_{{\R}^n} |\lambda_i|f d\gamma\Big)^2 C(r, M)^{\frac{1-\theta}{r\theta}};
\end{equation*}
as $\theta$ depends on $r$, let $\tilde C(r, M):=\frac{1}{c}C(r, M)^{\frac{1-\theta}{r\theta}}$. Moreover, set $s:=\frac{1}{\theta}>2$ and $a_i:=\int_{{R}^n} |\lambda_i|^2 f d\gamma$ so that
$$\sum_{i} a_i^s \leq \tilde C(r, M) \delta(f).$$ Now by H\"{o}lder,
\begin{align}
\Big(\sum_i a_i \Big)^s &\leq \Big(\sum_i a_i^s\Big)n^{s-1}  \label{csss}  \\
&\leq\tilde C(r, M) n^{s-1} \delta(f)  \nonumber.
\end{align}


\noindent Next,
\begin{equation} \label{bzxx1}
\sum_i a_i=\int_{{R}^n} \sum_{i} |\lambda_i|^2 fd\gamma =  \int_{{R}^n} ||D(T(x) - x)||_{HS}^2 \ fd\gamma,
\end{equation}
where $||\cdot||_{HS}$ denotes the Hilbert-Schmidt norm. As in the proof of Theorem \ref{logsob}, by applying Poincar\'{e} (Corollary \ref{cor})  to $T^i(x)-x_i$, 
$$\int_{{R}^n} |T(x) - x)|^2 \ d\mu_f (x) \leq C(\mu_f)\int_{{R}^n} ||DT-Id||_{HS}^2 d\mu_f(x).$$
Moreover, by combining this information with (\ref{bzxx1}) and (\ref{csss}) we obtain $$W_2^2(f d\gamma, d\gamma) = \int_{{R}^n} |T(x) - x)|^2 \ d\mu_f (x) \leq C(\mu_f)\Big(C(r, M) n^{s-1}\Big)^{\frac{1}{s}} \delta(f)^{\frac{1}{s}}.$$ Now $\lambda_1 \geq \epsilon$ since $f=e^{-h} \in \mathcal{\tilde F}(\epsilon, M)$ and so it follows from Corollary \ref{cor} that $C(\mu_f) \leq \frac{1}{\epsilon}$. This completes the proof and one may take $\beta:=\frac{1}{2s}<\frac{1}{4}$.


\end{proof}

\begin{rem}
Note that the dimensional dependence came into play when we utilized H\"{o}lder's inequality in (\ref{csss}). Indeed, if $a_i \thickapprox constant$, then $$\frac{\sum_i a_i^s}{\Big(\sum_i a_i\Big)^s} \thickapprox \frac{1}{n^{s-1}}.$$ Perhaps a different method may remove the dimension dependence; as we have seen in Theorem \ref{logsob}, this is possible under certain hypotheses (e.g. when one restricts the eigenvalues to be in $L^\infty$).
\end{rem}

\subsection{Sharpness}
\label{subsect: sharp}
In what follows, we show that the $\frac{1}{2}$ exponent in Theorem \ref{logsob} is sharp by considering a family of rescaled Gaussians. First, we recall some basic facts: given $\mu \in \R^n$ and a symmetric, positive-definite matrix $\Sigma$, by setting $$\displaystyle f(x) = \mathcal{N}(\mu, \Sigma) = \frac{1}{(2\pi)^{n/2} (\det \Sigma)^{1/2}} \exp \bigg( - \frac{1}{2} (x-\mu)^T \Sigma^{-1}(x-\mu) \bigg),$$ we have $$ \int_{\R^n} f(x) \ dx =1 $$ and $$\int_{\R^n} x f(x) \ dx = \mu.$$
So if we define $$f_a (x):= (2a + 1)^{n/2} e^{-a|x|^2},$$ since $$\displaystyle \frac{1}{(2\pi)^{n/2}} f_a (x) e^{-\frac{|x|^2}{2}} = \bigg(\frac{2a + 1}{2\pi}\bigg)^{n/2} e^{-(a+\frac{1}{2})|x|^2} = \mathcal{N}\bigg(0, \frac{1}{2a+1}I\bigg),$$ we readily obtain $$\int_{\R^n} f_a(x) \ d\gamma(x) =1$$ and $$\int_{\R^n} x f_a(x) \ d\gamma(x) = 0.$$ In particular, given $\epsilon>0$ and $M>0$, we have that for all $a>0$ small enough, $f_a \in \mathcal{F}(\epsilon, M)$. Moreover, $$f_a \log f_a =  \log (2a + 1)^{n/2} f_a  - a  |x|^2 f_a,$$ so
$$\int_{\R^n} f_a \log f_a \ d\gamma = \log (2a + 1)^{n/2}   - a \int_{\R^n} |x|^2 f_a \ d\gamma,$$ and integrating by parts yields
$$\int_{\R^n} |x|^2 f_a \ d\gamma = n \int_{\R^n} f_a \ d\gamma + \int_{\R^n}  x \cdot \grad f_a \ d\gamma = n - 2a \int_{\R^n} |x|^2 f_a \ d\gamma,$$ which implies $$\int_{\R^n} |x|^2 f_a \ d\gamma = \frac{n}{2a+1}.$$ Thus, we may write the entropy as:
\begin{equation} \label{ent1}
Ent (f_a) = \int_{\R^n} f_a \log f_a \ d\gamma = \log (2a + 1)^{n/2} - \frac{na}{2a+1}.
\end{equation}
Moreover, the Fisher information of $f_a$ is given by:
\begin{equation} \label{fishy}
\frac{1}{2} I(f_a)= \frac{1}{2} \int_{\R^n} \frac{|\grad f_a|^2}{f_a} \ d\gamma = 2 a^2 \int_{\R^n} |x|^2 f_a \ d\gamma = \frac{2na^2}{2a+1}.
\end{equation}

\noindent It is also not difficult to compute the Wasserstein distance between the two Gaussians $\mathcal{N}\big(0, \frac{1}{2a+1}\big)$ and $\mathcal{N}(0, 1)$:
\begin{equation} \label{wass}
W_2\bigg( \mathcal{N}\bigg(0, \frac{1}{2a+1}I \bigg), \mathcal{N}(0, I) \bigg)^2 = n\bigg(  \frac{1}{\sqrt{2a+1}} -1 \bigg)^2.
\end{equation}

\noindent Therefore, by utilizing (\ref{ent1}), (\ref{fishy}), and (\ref{wass}) we deduce

$$\frac{\delta(f_a)^{1/2}}{W_2(f_a d\gamma, d\gamma)}  = \frac{\sqrt{na-\frac{n}{2} \log (2a+1) }}{\sqrt{n}\big(1 - \frac{1}{\sqrt{2a+1}}\big)} = \frac{\sqrt{a-\frac{1}{2} \log (2a+1) }}{1- \frac{1}{\sqrt{2a+1}}},$$ and repeated applications of l' H\^opital's rule yields that as $a \rightarrow 0$,
\begin{equation} \label{sharp}
\frac{\delta(f_a)^{1/2}}{W_2(f_a d\gamma, d\gamma)} \to 1.
\end{equation}
Since $\delta(f_a) \rightarrow 0$ as $a \rightarrow 0$,  (\ref{sharp}) implies that the exponent $\frac{1}{2}$ may not be replaced by something larger.

\section{Controlling the entropy}
\label{S4}

As an application of Corollary \ref{logsob0}, we show how to obtain bounds on the entropy in terms of the deficit and barycenter. Let $d\gamma$ be the Gaussian measure and suppose that for a suitable class of functions $f$ we have an estimate of the form:
\begin{equation} \label{e_11}
W_2(f d\gamma, d\gamma) \leq C \delta^\alpha(f),
\end{equation}
for some $\alpha \in (0,\frac{1}{2}]$.
Thanks to Otto-Villani \cite{OV} (see also \cite[Corollary 3]{Co}), we know that if $f$ has unit mass with respect to the Gaussian,
\begin{equation} \label{e_22}
Ent(f) \leq W_2(fd\gamma, d\gamma)\sqrt{I(f)}-\frac{1}{2}W_2^2(fd\gamma, d\gamma);
\end{equation}
this inequality is known as the HWI inequality.
\begin{lem} \label{app}
Suppose (\ref{e_11}) holds for a suitable class of functions. Then $$Ent(f) \leq \tilde C (\delta^{\frac{1}{2}+\alpha}(f)+\delta^{2\alpha}(f)).$$
\end{lem}

\begin{proof}
We simplify the notation in an obvious way. First, note that since all the quantities are non-negative $$W \sqrt{I} =W \sqrt{2(\delta+E)} \leq \sqrt{2}W (\sqrt{\delta}+\sqrt{E})=\sqrt{2}W \sqrt{\delta}+\sqrt{2}W\sqrt{E},$$ (using $\sqrt{a+b}\leq \sqrt{a}+\sqrt{b}$). Now, $$W \sqrt{E} = \frac{W}{\epsilon} (\epsilon\sqrt{E}) \leq \frac{W^2}{2\epsilon^2}+\frac{\epsilon^2}{2}E,$$ (using $ab \leq \frac{1}{2}(a^2+b^2)$) and thus, $$W\sqrt{I} \leq \sqrt{2}W \sqrt{\delta}+\sqrt{2}\Big(\frac{W^2}{2\epsilon^2}+\frac{\epsilon^2}{2}E\Big).$$ Hence, an application of (\ref{e_22}) yields $$E \leq \sqrt{2}W \sqrt{\delta}+\sqrt{2}\Big(\frac{W^2}{2\epsilon^2}+\frac{\epsilon^2}{2}E\Big)-\frac{1}{2}W^2=\sqrt{2}W\sqrt{\delta}+\Big(\frac{\sqrt{2}}{2\epsilon^2} -\frac{1}{2}\Big )W^2+\frac{\sqrt{2}{\epsilon^2}}{2}E;$$ using (\ref{e_11}) we have $$(1-c\epsilon^2)E \leq C(\delta^{\frac{1}{2}+\alpha}+\delta^{2\alpha}).$$ Picking $\epsilon$ sufficiently small completes the proof.
\end{proof}

\begin{cor} \label{bet1}
There exists an explicit dimensionless constant $\bar C=\bar C(\epsilon, M)>0$ so that for all $f \in \mathcal{F}(\epsilon, M)$, $$Ent(f) \leq \bar C\delta(f)+\frac{1}{2}|\mu(f)|^2,$$ where $\mu(f)$ is the barycenter of $f$ with respect to $d\gamma$. 
\end{cor}

\begin{proof}
Given $f \in \mathcal{F}(\epsilon, M)$, let $$\hat f(x) := f(x)e^{-\big(\langle \mu(f), x\rangle + |\mu(f)|^2/2 + \log(||f||_{L^1(d\gamma)})\big)},$$ where $\mu(f)$ is the barycenter of $f$ with respect to $d\gamma$. Note that by Theorem \ref{logsob}, $$W_2\big(\hat f d\gamma, d\gamma\big) \leq C\delta\big(\hat f/||f||_{L^1(d\gamma)}\big)^{\frac{1}{2}}.$$ Hence, Lemma \ref{app} implies $$Ent\big(\hat f\big /||f||_{L^1(d\gamma)} \big) \leq \bar C \delta\big(\hat f /||f||_{L^1(d\gamma)}\big),$$ and so $$Ent(\hat f) \leq \bar C \delta(\hat f).$$ From the proof of Theorem \ref{logsob}, we know $\delta\big(\hat f\big)=\delta(f)$ and $Ent\big(\hat f\big)=Ent(f)-\frac{1}{2}|\mu(f)|^2$, and this yields the result.

\end{proof}

\begin{rem} \label{bet2}
Let $f \in \mathcal{F}(\epsilon, M)$ and consider $\hat f$ as in the proof of Corollary \ref{bet1}. Since $\hat f$ has zero barycenter,  $$Ent(\hat f) \leq \bar C\delta(\hat f).$$ Thus, $$Ent(\hat f) \leq \frac{\bar C}{2(\bar C+1)} I(\hat f);$$ now, as $C:=\frac{\bar C}{2(\bar C+1)}<\frac{1}{2}$, this improves the constant in the log-Sobolev inequality for functions in $\mathcal{F}(\epsilon, M)$ with unit mass and zero barycenter with respect to $d\gamma$. More generally, since $Ent(\hat f)=Ent(f)-\frac{1}{2}|\mu(f)|^2$ and $I\big(\hat f\big)=I(f)-|\mu(f)|^2$ (see e.g. the proof of Theorem \ref{logsob}), it follows that 
\begin{equation} \label{quan}
Ent(f) \leq CI(f)+((1/2)-C)|\mu(f)|^2.
\end{equation}
We note that as $M \rightarrow \infty$ or $\epsilon \rightarrow 0$, $\bar C=\bar C(\epsilon, M) \rightarrow \infty,$ and so when we enlarge our function space in this way, $C \rightarrow \frac{1}{2}$ -- the sharp log-Sobolev constant for general functions. Indeed, (\ref{quan}) measures the improvement in the log-Sobolev constant for the class $\mathcal{F}(\epsilon, M)$ in terms of the barycenter.   
\end{rem}

\noindent \textbf{Acknowledgments.} We wish to thank Alessio Figalli for suggesting the problem and for lively discussions on this topic. E. Indrei was supported by a departmental fellowship for graduate studies at the University of Texas at Austin. D. Marcon was supported by the UT Austin-Portugal partnership through the FCT doctoral fellowship SFRH/BD/33919/2009. Most of this work was completed while the authors were participating in the program ``Concentration month on nonlinear elliptic PDEs" at the University of Chicago -- the excellent research environment is kindly acknowledged.

%


\pagebreak 
\signei

\signdm

\end{document}